\documentclass[12pt]{article}
\usepackage[left=1in,top=1in,right=1in,bottom=1in,nohead]{geometry}

\usepackage{url,subfig,latexsym,amsmath,amssymb, amsfonts,enumerate,
  epsfig, graphics,times, amsthm,mathtools}
\usepackage[inline]{enumitem}
\usepackage{color}
\usepackage{nameref}
\usepackage{bbm}
\usepackage[sort&compress,numbers]{natbib}
\usepackage{tikz}      
\usetikzlibrary{shapes,automata,positioning,arrows,fit}
\usepackage[colorlinks]{hyperref}

\newtheorem{thm}{Theorem}[section]

\newtheorem{lem}[thm]{Lemma}

\newtheorem*{thm*}{Theorem}

\theoremstyle{definition}
\newtheorem{defn}{Definition}[section]

\newtheorem{example}{Example}[section]

\newcommand{\lemref}[1]{Lemma~\ref{#1}}

\newcommand{\dsum}{\displaystyle \sum\limits}

\newcommand{\dprod}{\displaystyle \prod\limits}

\def\S{\mathcal S}
\def\C{\mathcal C}
\def\Re{\mathcal R}
\def\R{\mathbb R}

\def\Z{\mathbb Z}

\def\Td1{T^{D,1}_{\{x_n\}}}
\def\Ts1{T^{S,1}_{\{x_n\}}}

\def\Tsinf{T^{S,\infty}_{\{x_n\}}}

\def\lmt{\lim_{n\rightarrow \infty}}

\def\Tid1{T^{D,1}_{\{x_n+\zeta_i\}}}
\def\Tis1{T^{S,1}_{\{x_n+\zeta_i\}}}
\def\Tjs1{T^{S,1}_{\{x_n+\zeta_j\}}}
\def\Tjd1{T^{D,1}_{\{x_n+\zeta_j\}}}

\newcommand{\state}{\mathbb{S}_{x_0}}

\newcommand{\reply}[1]{{#1}}

\newcount\colveccount
\newcommand*\colvec[1]{
        \global\colveccount#1
        \begin{pmatrix}
        \colvecnext
}
\def\colvecnext#1{
        #1
        \global\advance\colveccount-1
        \ifnum\colveccount>0
                \\
                \expandafter\colvecnext
        \else
                \end{pmatrix}
        \fi
}

\title{Stochastically modeled weakly reversible  reaction networks with a single linkage class} 

\author{
David F. Anderson\thanks{Department of Mathematics, University of
  Wisconsin, Madison, USA.  anderson@math.wisc.edu},
\and 
Daniele Cappelletti \thanks{Department of Biosystems Science and Engineering, ETH Zurich, Switzerland. daniele.cappelletti@bsse.ethz.ch}\and
Jinsu Kim\thanks{Department of Mathematics, University of
  California, Irvine, USA.  jinsu.kim@uci.edu}
  }

\begin{document}

 \tikzset{every node/.style={auto}}
 \tikzset{every state/.style={rectangle, minimum size=0pt, draw=none, font=\normalsize}}
 \tikzset{bend angle=7}

\maketitle

\begin{abstract}
\reply{It has been known for nearly a decade that deterministically modeled reaction networks that are weakly reversible and consist of a single linkage class have trajectories that are bounded from both above and below by positive constants (so long as the initial condition has strictly positive components).}
It is conjectured that the stochastically modeled analogs of these systems are positive recurrent.  We prove this conjecture in the affirmative under the following additional assumptions:  (i) the system is binary and (ii) for each species, there is a complex (vertex in the associated reaction diagram) that is a multiple of that species. To show this result, a new proof technique is developed in which we study the recurrence properties of the $n-$step embedded discrete time Markov chain. 
\end{abstract}

\section{Introduction}

Reaction networks are directed graphs that can be used to describe the general dynamics of population models.  Their use is widespread in population ecology, chemistry, and cell biology. An example of reaction network is
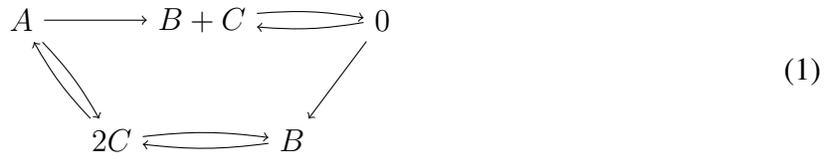
\begin{equation}\label{eq:example}
  \begin{tikzpicture}[baseline={(current bounding box.center)}, scale=0.8]
   \node[state] (1) at (0,2)  {$A$};
   \node[state] (2) at (3,2)  {$B+C$};
   \node[state] (3) at (6,2)  {$0$};
   \node[state] (4) at (4.5,0)  {$B$};
   \node[state] (5) at (1.5,0)  {$2C$};
   \path[->]
    (1) edge[bend left] node {} (5)
        edge node {} (2)
    (2) edge[bend left] node {} (3)
    (3) edge[bend left] node {} (2)
        edge node {} (4)
    (4) edge[bend left] node {} (5)
    (5) edge[bend left] node {} (4)
        edge[bend left] node {} (1);
  \end{tikzpicture}
 \end{equation}
where the possible transformations of \emph{species} $A$, $B$, and $C$ are depicted: in the reaction $A\to B+C$, a single individual of species $A$ can be transformed into one individual of species $B$ and one individual of species $C$. In the reaction $B+C\to 0$, individuals of species $B$ and $C$ can annihilate each other. In the opposite reaction $0 \to B+C$, one individual of $B$ and one individual of $C$ are introduced in the system at the same time, and so on. Typically, in biochemistry the species are thought of as different chemical components that react with each other. The nodes of the graph are termed \emph{complexes}. Different mathematical models can be associated with such a network in order to describe the dynamics of the system of interest. Regardless of the modeling choice, a natural question is the following: when is it possible to prove theorems that relate the structure of the reaction diagrams, which are developed by domain scientists, to general dynamical properties of the associated mathematical models? The subfield of mathematics that works on this question is often termed ``chemical reaction network theory'', and has a history dating back to at least the early 1970s \cite{Feinberg72, Horn72, HornJack72}.  In particular, there has been an enormous amount of work relating the possible dynamics of the associated deterministic mass-action system with the basic structure of the reaction network itself; see the following for just a subset \cite{Craciun2006_PNAS,AndGAC_ONE2011,GMS13b, CDSS09, CW:intermediates, complex_invariants}. 

There has also been work in this sense for stochastic dynamics, which are typically modeled through a continuous time Markov chain as described in Section \ref{sec:basic}. 
This paper is concerned with providing structural conditions on the reaction network that guarantee the associated continuous time Markov chain is positive recurrent. This information is precious for practical purposes: as an example, in biological models it is often the case that different chemical species evolve over different time scales. It is known that under certain assumptions, the dynamics of the slower species can be approximated by averaging out the dynamics of the faster subsystem, which is considered to be at stationary regime \cite{KangKurtz2013,PP:scaling,Ball06}. It is therefore important to understand whether the fast subsystem admits a stationary regime in the first place. Once the existence of the stationary distribution is proven, this can be approximated by simulations, or by finite state-space projections \cite{MK:finite, GMK:finite, AE2019}. Another example concerns synthetic biology: in this field, biological circuits are designed and implemented in order to control cellular behavior. Often, it is rigorously proven that such circuits serve the desired purpose provided that a stationary regime can be reached \cite{BGK:antithetic, BGK:antithetic_variance}. Finally, from a mathematical standpoint it is  inherently interesting to understand the recurrence properties of this large class of  models, which are used ubiquitously in the science literature.


Despite the importance of understanding whether a stationary regime exists, checking for positive recurrence in the setting of reaction networks is usually a hard task, with the exception of a few classes of models for which positive recurrence has been shown \cite{AndProdForm, AC:product, CW:product, JH:mono, AK2018, ACKK2018}. Probably the most important open conjecture related to stochastic models in the field is the following: if the reaction network is \emph{weakly reversible} (i.e., if each connected component of the reaction graph is strongly connected; see Definition \ref{def:WR}), then the associated Markov model is positive recurrent.  There is a large amount of evidence that this conjecture is true.  \reply{In particular, in \cite{AndProdForm} it was shown that if the deterministic model associated with a particular network is \textit{complex balanced} (meaning that at equilibrium the total flux into each complex is equal to the total flux out), then the stochastic model associated with that same network admits a stationary distribution that is a product of Poissons.}
Followup work in \cite{ACKK2018} showed that these models are non-explosive, implying they are positive recurrent.  A key fact is that for \emph{every} weakly reversible reaction network, there is a choice of rate constants that makes the model complex balanced.  Hence, if there is a weakly reversible model that is not positive recurrent, the lack of recurrence must be a property of the specific choice of rate constants, and not of the network structure, something most researchers find unlikely.  


In this paper, we do not consider the full class of weakly reversible networks, but instead restrict ourselves to those networks that are weakly reversible and have a single connected component (termed a ``linkage class'' in the literature).  This particular set of networks was first  studied in the deterministic context in \cite{AndBounded_ONE2011,AndGAC_ONE2011}, where it was shown that 
they necessarily have bounded, persistent trajectories. That is, for each trajectory $x(t)\in \R^d_{\ge 0}$, the expression $\sum_{i=1}^d |\ln(x_i(t))|$ is uniformly bounded in time. This subclass of models was studied in the stochastic context in \cite{AK2018}, where positive recurrence was proven under a few additional assumptions on the network, which we will not list here due to their technical nature.  

We prove in Theorem \ref{thm:main} that models that are weakly reversible and consist of a single linkage class are positive recurrent under two additional structural assumptions: 
(i) the system is \emph{binary} \reply{(see Definition \ref{def:binary})} and (ii) for each species, there is a complex that is a multiple of that species. 
 As an example, these assumptions are met in \eqref{eq:example}.

Providing a new class of models for which positive recurrence is guaranteed is the first major contribution of this paper.  The second is the method of proof, which we believe will be useful in the study of a significantly wider class of models.  In particular, most previous results related to positive recurrence of stochastically modeled reaction networks that the authors are aware of utilized the Foster-Lyapunov criteria related to the continuous time Markov chain itself.  \reply{In this article, we prove positive recurrence of the continuous time chain by studying the recurrence properties of the discrete time  Markov chain that arises by observing the continuous time Markov chain after a fixed number of jumps.}

This paper has the following outline. In Section \ref{sec:notation}, we briefly provide a catalog of notation we use throughout this paper. In Section \ref{sec:basic}, we provide the basic mathematical objects we will study. 
In Section \ref{sec:main_result} we state our main result and in Section \ref{sec:tiers} we introduce the concept of tier sequences, which is fundamental to show our main result. The connection between tiers and positive recurrence is detailed in Section~\ref{sec:positive_rec}, at the end of which the proof of the main structural result is given. 

\section{Notation}
\label{sec:notation}

\reply{We denote by $\R$, $\R_{>0}$, and $\R_{\ge 0}$ the real, positive, and non-negative real numbers, respectively. Similarly, we denote by $\Z$, $\Z_{>0}$, and $\Z_{\ge 0}$ the integers, positive integers, and non-negative integers.}  Any sum of the form $\sum_{i = 1}^0 a_i$ is taken to be zero, as is usual.

Given two vectors $x,z\in\R^d$, we write $x\geq y$ if $x_i\geq z_i$ for all $1\leq i\leq d$. \reply{When $x,y \in \R^d_{\ge0}$, we denote by $x^y$ the vector whose $i$th component is given by $x_i^{y_i}$, with the usual convention that $0^0$ is taken to be  1.  Moreover, we denote by $x\vee 1$ the vector in $\R_{>0}^d$ whose $i$th component is $\max\{x_i, 1\}$.}

 Given a stochastic process $\{X(t):t\in\R_{\geq0}\}$ with state space $\Gamma$, and $x\in\Gamma$, we denote by
\begin{equation}\label{eq:Px}
    P_x(X(t)\in A)=P(X(t)\in A | X(0)=x)\quad\text{and}\quad E_x[f(X(t))]=E[f(X(t))|X(0)=x]
\end{equation}
for all measurable sets $A$, all measurable functions $f\colon\Gamma\to\R$ (with respect to the $\sigma-$algebra generated by $X$), and all $t\in\R_{>0}$.

We assume basic knowledge of Markov chain theory. In particular, we will not define what the generator of a Markov chain is, what it means for a Markov chain to have a stationary distribution, and what it means for a Markov chain to be positive recurrent, transient, or explosive. These concepts can be read from \cite{Kurtz86, NorrisMC97}. We will also assume the Foster-Lyapunov function criterion for continuous time and discrete time Markov chains is known to the reader \cite{MT-LyaFosterIII}.

\section{Reaction networks and stochastic dynamics}\label{sec:basic}
In Section \ref{sec:reactionnetworks}, we formally introduce reaction networks.  In Section \ref{sec:dynamics}, we introduce the associated stochastic models.  \reply{In Section \ref{sec:embedded}, we introduce the embedded discrete time Markov chain for our model.}

\subsection{Reaction networks}
\label{sec:reactionnetworks}

A reaction network describes the set of possible interactions among constituent chemical species, and is defined as follows. 

\begin{defn}\label{def:21}
\emph{A  reaction network} is given by a triple of finite sets $(\S,\C,\Re)$ where:
\begin{enumerate}[label=(\roman*)]
\item The set $\S=\{S_1,S_2,\cdots,S_d\}$ is a set of $d$ symbols, called the \emph{species} of the network;
\item The set $\C$ is a set of linear combinations of species on $\Z_{\geq 0}$, called \emph{complexes};
\item The set $\Re$ is a \reply{non-empty} subset of $\C\times\C$, \reply{called \emph{reactions}}, whose elements $(y,y')$ are denoted by $y\to y'$. The cardinality of $\Re$ is denoted by $r$.
\end{enumerate}
\end{defn}
Given a reaction $y\to y'\in\Re$, the complexes $y$ and $y'$ are termed the \emph{source} and \emph{product} complex of the reaction, respectively. It is common to assume that $y\to y\notin\Re$ for all $y\in\C$, and we will do so in the present paper.
Throughout, we abuse notation slightly by letting a complex $y\in\C$ denote both a linear combination of the form
$$y=\sum_{i=1}^d y_iS_i$$
as defined in Definition~\ref{def:21}, and the vector whose $i$th component is $y_i$,  i.e.~$y=(y_1,y_2,\cdots,y_d)^T \in \mathbb{Z}^d_{\ge 0}$. For example, if $y=2S_1+S_2$ then $y$ will also be used to denote the vector $(2,1,0,0,\dots,0) \in \Z^d_{\ge 0}$. We denote by $0$  the complex $y$ for which $y_i=0$ for each $i$.

The directed graph whose nodes are $\C$ and whose \reply{directed} edges are $\Re$ can be associated to a reaction network. Such a graph is called the \emph{reaction graph} of the reaction network.

\begin{defn}\label{def:WR}
Let $(\S,\C,\Re)$ be a reaction network.  The connected components of the associated reaction graph are termed \emph{linkage classes}.  We say that a linkage class is \emph{weakly reversible} if it is strongly connected, i.e., if for any two complexes $y_1,y_2$ in the linkage class there is a directed path from $y_1$ to $y_2$ and a directed path from $y_2$ to $y_1$. 

If all linkage classes in a reaction network are weakly reversible, then the reaction network is said to be weakly reversible. 
\end{defn}

\begin{example}\label{ex11}
Consider the reaction network with the following reaction graph:
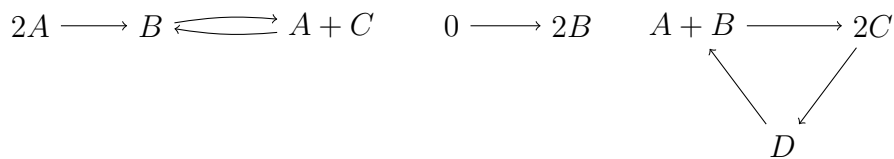
\begin{equation*}
   \begin{tikzpicture}[baseline={(current bounding box.center)}, scale=0.8]
   \node[state] (1) at (0,2)  {$2A$};
   \node[state] (2) at (2,2)  {$B$};
   \node[state] (3) at (5,2)  {$A+C$};
   \node[state] (4) at (7,2)  {$0$};
   \node[state] (5) at (9,2)  {$2B$};
   \node[state] (6) at (11,2)  {$A+B$};
   \node[state] (7) at (14,2)  {$2C$};
   \node[state] (8) at (12.5,0)  {$D$};
   \path[->]
    (1) edge node {} (2)
    (2) edge[bend left] node {} (3)
    (3) edge[bend left] node {} (2)
    (4) edge node {} (5)
    (6) edge node {} (7)
    (7) edge node {} (8)
    (8) edge node {} (6);
  \end{tikzpicture}
\end{equation*}
This network has three linkage classes.  The right-most linkage class is weakly reversible, whereas the other two are not.  Because not all linkage classes are weakly reversible, the network is not weakly reversible.
\hfill $\triangle$
\end{example}

The main result of this paper pertains to binary reaction networks. 

\begin{defn}\label{def:binary}
A reaction network $(\S,\C,\Re)$ is called \emph{binary}  if $\sum_{i=1}^d y_i \le 2$ for all $y \in \C$.
\end{defn}
\reply{Most networks found in the science literature are binary.}

\subsection{Stochastic model}
\label{sec:dynamics}

In this section we introduce a stochastic dynamical model associated with reaction networks, which is often utilized in biochemistry if few molecules of the different species are present and stochastic fluctuations cannot be ignored.
First, we associate with each reaction $y\to y'$ a \emph{rate function} $\lambda_{y\to y'}\colon \Z^d\to \R_{\geq0}$, expressing the propensity of a reaction to occur. We further assume that a reaction $y\to y'$ cannot take place if not enough reacting molecules are present, which is expressed by the condition $\lambda_{y\to y'}(x)>0$ only if $x\geq y$. A popular choice of rate functions is given by \emph{mass action kinetics}, where it is assumed that the propensity of a reaction to occur is proportional to the number of combinations of the present molecules that can give rise to the reaction. Specifically, mass action rate functions take the form
\begin{align}\label{mass}
\lambda_{y\rightarrow y'}(x)= \kappa_{y\rightarrow y'} \lambda_y(x), \quad\text{where}\quad \lambda_y(x)=\prod_{i=1}^d \frac{x_i !}{(x_i-y_{i})!}\mathbf{1}_{\{x_i \ge y_{i}\}},
\end{align}
for some positive constants $\kappa_{y\to y'}$ called \emph{rate constants}, \reply{and where $\mathbf{1}_{\{x_i \ge y_{i}\}}$ is the indicator function on the event $\{x_i \ge y_i\}$}. A reaction network together with a choice of rate constants $(\S, \C, \Re, \kappa)$ is called a \emph{mass action system}. We say that a mass action system is weakly reversible, binary, or has a single linkage class if the related reaction network does.

For every $t\geq0$, we denote by $X(t)$ a vector in $\Z^d_{\geq0}$ whose $i$th entry represents the \reply{count} of species $S_i$ at time $t$. We assume that $X$ is a continuous time Markov chain over $\Z^d_{\geq0}$, with transition rate from state $x\in \Z^d_{\geq0}$ to state $z \in \Z^d_{\geq0}$ given by
\begin{equation}\label{eq:q}
    q(x, z)=\sum_{\substack{y\to y'\in \Re\\ y'-y=z-x}}\lambda_{y\to y'}(x).
\end{equation}
The generator $\mathcal{A}$ of the process $X$ is \reply{the operator whose action on a function $V:\Z^d_{\ge 0} \to \R$ is given by \cite{Kurtz86}}
\begin{align}\label{gen5}
 \mathcal{A}V(x) = \sum_{y \rightarrow y' \in \Re} \lambda_{y\rightarrow y'}(x)(V(x+y'-y)-V(x)).
\end{align}
Moreover, assume $X(0)=x_0$. We then define the set of \emph{reachable states} as
$$\state=\{x\in\Z^d_{\geq0}\,:\,P_{x_0}(X(t)=x)>0\quad\text{for some }t\geq0\},$$
where $P_{x_0}$ is as in \eqref{eq:Px}.

\subsection{The embedded discrete time Markov chain}
\label{sec:embedded}

Let the continuous time Markov chain $X$ be as defined in the previous section. In the present paper, we will utilize the embedded discrete time Markov chain of $X$, defined as follows \cite{NorrisMC97}.
\begin{defn}\label{defn:embedded}
Let $\tau_0=0$ and for $n \in \Z_{>0}$ let $\tau_n$ be the time of the $n$th jump of $X$. Then the \emph{embedded discrete time Markov chain of $X$} is the discrete time Markov chain $\{\tilde{X}_n\}_{n=0}^\infty$ defined by $\tilde X_n=X(\tau_n)$ for all $n\in\Z_{\geq0}$. 
\end{defn}

It is straightforward to show that the transition probabilities of the embedded discrete-time Markov chain $\tilde{X}$ are given by
\begin{equation}\label{eq:transition prob1}
P(\tilde{X}_1=z|\tilde{X}_0=x)=\frac{q(x,z)}{\overline{\lambda}(x)},
\end{equation} 
where $q$ is as in \eqref{eq:q} and
\begin{equation}\label{lambdabar}
    \overline{\lambda}(x)=\sum_{y\to y'\in \Re}\lambda_{y\to y'}(x).
    \end{equation}

The following result, which can be inferred from  \cite[Theorem 3.5.1]{NorrisMC97}, will be used.

\begin{thm} \label{thm : relation between conti and disc}
 Suppose that $\inf_{x\in \state}\overline{\lambda}(x) > 0$ and that $\tilde{X}$ restricted to $\state$ is positive recurrent. Then, $X$ restricted to $\state$ is non-explosive and positive recurrent. 
\end{thm}

The condition $\inf_{x\in \state}\overline{\lambda}(x) > 0$ holds when the process $X$ is associated with a weakly reversible mass action system, and the initial condition $X(0)=x_0$ is not an absorbing state. This is expressed by the following result, which can be derived from \cite[Lemma 4.6]{PCK2014}.  Note that the word ``recurrent'' in \cite{PCK2014} is not to be intended in the usual stochastic sense, which is considered here.

\begin{lem}\label{lem : minimum of intensities}
Let $(\S,\C,\Re, \kappa)$ be a weakly reversible mass action system, and assume $\overline\lambda(x_0)>0$. Then $\inf_{x\in \state}\overline{\lambda}(x) > 0$.
\end{lem}

\section{Main result}
\label{sec:main_result}

Our main result is the following.
\begin{thm}\label{thm:main}
 Consider a weakly reversible, binary mass action system $(\S,\C,\Re, \kappa)$ that has a single linkage class, and let $X$ be the associated continuous time Markov chain. Assume that for each $S\in \S$, $\{S, 2S\} \cap \C\neq \emptyset$. Then $X$ is positive recurrent.
\end{thm}

To prove Theorem \ref{thm:main} we will use Foster-Lyapunov functions, the notion of tiers \reply{(explained and defined in the next section)}, and properties of \reply{an associated $n-$step} embedded discrete time Markov chain. Specifically, consider the entropy-like function $V\colon \R^d_{\geq0}\to \R_{\geq0}$ defined by \begin{equation}\label{eq:MainLyapunov}
V(x) = \sum_{i=1}^d \Big( x_i(\ln{x_i}-1)+1\Big),   
\end{equation}
with the convention $0 \ln{0}=0$. 
Notably, this function has been shown to be a Lyapunov function for the deterministic models of \emph{complex balanced} mass action systems \cite{HornJack72, Horn72}, and\reply{, in connection with \emph{tier sequences}, for those models that are weakly reversible and have a single linkage class} \cite{AndBounded_ONE2011,AndGAC_ONE2011}. It has been further utilized to study the stability of stochastic models of mass action systems in different works, including \cite{ACGW:lyapunov, AK2018, ACKN2018, dembo2018}. \reply{In this paper}, we will show that $V$ can be utilized as a Foster-Lyapunov function for the $n-$step discrete time embedded chain, \reply{for a specially chosen positive integer $n$,} even though it may not be a Foster-Lyapunov function for the continuous time Markov chain $X$ itself.

In the following example we show how Theorem \ref{thm:main} can be applied, and how the function $V$ cannot be readily used as Foster-Lyapunov function for $X$.
\begin{example}\label{ex}
Consider the mass action system
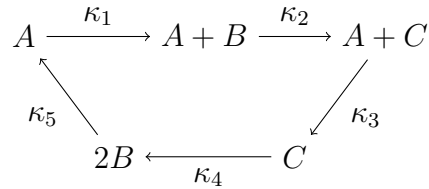
\begin{equation*}
   \begin{tikzpicture}[baseline={(current bounding box.center)}, scale=0.8]
   \node[state] (1) at (0,2)  {$A$};
   \node[state] (2) at (3,2)  {$A+B$};
   \node[state] (3) at (6,2)  {$A+C$};
   \node[state] (4) at (4.5,0)  {$C$};
   \node[state] (5) at (1.5,0)  {$2B$};
   \path[->]
    (1) edge node {$\kappa_1$} (2)
    (2) edge node {$\kappa_2$} (3)
    (3) edge node {$\kappa_3$} (4)
    (4) edge node {$\kappa_4$} (5)
    (5) edge node {$\kappa_5$} (1);
  \end{tikzpicture}
\end{equation*}
where each rate constant is placed next to its corresponding reaction. The system is weakly reversible and has a single linkage class, and $A,2B,C\in \C$. Then, Theorem \ref{thm:main} applies and the associated Markov process $X$ is positive recurrent for any choice of rate constants. To the best of our knowledge, there are no other results in the literature that allows one to quickly reach this conclusion for this model, and finding a tailored Foster-Lyapunov function may not be easy. In particular, the function $V$ as defined in \eqref{eq:MainLyapunov} is not a Foster-Lyapunov function for $X$. To show this, it suffices to show that there exists an initial condition $X(0)=x_0$ and a sequence of points $\{x_n\}$ in $\state$ whose norm tends to infinity and for which $$\limsup_{n\to\infty}\mathcal{A}V(x_n)>0.$$
To this aim, consider $x_0=(1,0,0)$ and the sequence defined by $x_n=(n,1,0)$. \reply{Note that by sequential occurrences of reactions 1 and 5, we may conclude that $x_n \in \state$ .  Moreover,} we have
$$\mathcal{A}V(x_n)=\kappa_1n(2\ln{2}-1),$$
which tends to infinity as $n$ tends to infinity.
 \hfill $\triangle$
\end{example}

\reply{
We will give here a very high-level motivation for the arguments that will follow. First, loosely speaking, the assumption that either $S \in \C$ or $2S \in \C$, ensures that for any sequence of points $x_n$ going to $\infty$ one of the source complexes with largest intensities is of the form $S$ or $2S$.  The intensity functions for these reactions are ``smooth'' away from the origin, in the sense that they can not change dramatically due to a single instance of a reaction.  This is unlike, say, the complex $S_1+S_2$, whose intensity can change dramatically through small changes, as seen by considering the values $(n,1)$ and $(n,0)$.  Second, the weak reversibility assumption is helpful as it guarantees that each product complex is also a source complex. Hence, whenever a reaction has the largest intensity, the reactions starting from the product complex are either of the same order of magnitude, or smaller.  By following a sequence of such reactions (leading to the analysis of the $n-$step chain), we are guaranteed to eventually find a source complex whose reaction rates are of lower order of magnitude, and we will show how this guarantees positive recurrence by the Foster-Lyapunov criterion.}


\section{Tiers}
\label{sec:tiers}

In this section, we introduce both D-type and S-type tier sequences, which have now appeared in the literature a number of times \cite{ACKN2018,AndBounded_ONE2011,AndGAC_ONE2011, AK2018},  and state some known relevant results.  
Typically, tier sequences are used in connection with the function $V$ as described in \eqref{eq:MainLyapunov}. The usual strategy consists of proving that along any tier sequence $\{x_n\}$ we have
$$\limsup_{n\to\infty}\mathcal{A} V(x_n)\leq -1,$$
implying positive recurrence by the Foster-Lyapunov criterion \cite{MT-LyaFosterIII} and from the fact that any sequence of states has a tier subsequence.  

The above approach does not always work. In particular, Example~\ref{ex} shows how it may fail for models in the class described in Theorem~\ref{thm:main}. Hence, in order to prove Theorem~\ref{thm:main}, we will introduce a new concept of tiers along sequences of paths in Section~\ref{sec:along_path}. With this new concept, we will be able to prove that even if $\limsup_{n\to \infty} \mathcal{A}V(x_n) >0$ for a particular tier sequence $\{x_n\},$ positive recurrence can still be shown by consideration of how $V$ changes over a fixed number of transitions.

\subsection{Definitions}\label{sec:tier_def}

\reply{In this section, we formally introduce the concept of ``tiers.'' Tiers are partitions of the set of complexes, and they are based on the different orders of magnitude of the different rate functions along a sequence of points. We define two different types of tiers, specifically D-type and S-type tiers. We will give  intuition on the role of these objects after their formal definitions.}

\begin{defn}\label{def31}
Let $( \S,\C,\Re)$ be a  reaction network and let $\{x_n\}$ be a sequence in $\Z^d_{\ge 0}$, with $n$ varying from 1 to $\infty$.  We say that $\C$ has a \emph{D-type partition along  $\{x_n\}$} if there exists a partition \reply{$\{T^{D,1}_{\{x_n\}},\dots, T^{D,P}_{\{x_n\}}\}$} of $\C$ such that:
\begin{enumerate}[label=(\roman*)]
\item if $y,y' \in T^{D,i}_{\{x_n\}}$, then 
\begin{equation*}
C=\lim_{n\rightarrow \infty} \frac{(x_n \vee 1)^{y}}{(x_n\vee 1)^{y'}} 
\end{equation*}
exists, and $0<C<\infty$;
\item if $y \in T^{D,i}_{\{x_n\}}$ and $y' \in T^{D,k}_{\{x_n\}}$ with $i < k$ then
\begin{align*}
\lim_{n\rightarrow \infty} \dfrac{(x_n \vee 1)^{y'}}{(x_n\vee 1)^{y}} = 0.
\end{align*}
\end{enumerate}
The sets of a D-type partition are called \emph{D-type tiers along $\{x_n\}$}.  We will say that $y$ is in a higher D-type tier than $y'$ along $\{x_n\}$ if $y \in T^{D,i}_{\{x_n\}}$ and $y' \in T^{D,j}_{\{x_n\}}$ with $i < j$. In this case, we will write $y \succ_D y'$. If $y$ and $y'$ are in the same D-type tier, then we will write $y\sim_D y'$. 
\end{defn} 



\begin{defn}\label{def32}
Let $(\S,\C,\Re)$ be a reaction network and let $\{x_n\}$ be a sequence in $\Z^d_{\ge 0}$, with $n$ varying from 1 to $\infty$. We say that $\C$ has a \emph{S-type partition along  $\{x_n\}$} if there exists a partition $\{T^{S,1}_{\{x_n\}}, \cdots , T^{S,P}_{\{x_n\}}, T^{S,\infty}_{\{x_n\}}\}$ of $\C$ such that:
\begin{enumerate}[label=(\roman*)]
\item $y \in T^{S,\infty}_{\{x_n\}}$ if and only if $\lambda_{y}(x_n) = 0$ for all $n\geq 1$; 
\item $y \in \bigcup_{i=1}^P T^{S,i}_{\{x_n\}}$  if and only if   $\lambda_y(x_n) \ne 0$ for all $n\geq1$;
\item if $y,y' \in T^{S,i}_{\{x_n\}}$, with $i \in \{1,\dots,P\}$, then 
\begin{align*}
C=\lim_{n\rightarrow \infty} \dfrac{\lambda_{y'}(x_n)}{\lambda_{y}(x_n)} 
\end{align*}
exists, and $0<C<\infty$;
\item if $y \in T^{S,i}_{\{x_n\}}$ and $y' \in T^{S,k}_{\{x_n\}}$ with $1\le i < k\le P$, then
\begin{align*}
\lim_{n\rightarrow \infty} \dfrac{\lambda_{y'}(x_n)}{\lambda_{y}(x_n)} = 0.
\end{align*}
\end{enumerate}
The sets of a S-type partition are called \emph{S-type tiers along $\{x_n\}$}.  We will say that $y$ is in a higher S-type tier than $y'$ along $\{x_n\}$ if $y \in T^{S,i}_{\{x_n\}}$ and $y' \in T^{S,j}_{\{x_n\}}$ with $i < j$. 
\end{defn}

\reply{The above tier structures make hierarchies for the complexes  with respect to the sizes of $(x_n \vee 1)^y$ and $\lambda_y(x_n)$, respectively, along a sequence $\{x_n\}$. Specifically, a D-type tier structure provides an ordering for the order of magnitudes of \emph{deterministic} mass-action rate functions of different reactions: the deterministic mass-action rate of the reaction $y_1\to y_1'$ is of a higher order of magnitude with respect to the rate of $y_2\to y_2'$, along a sequence of points $\{x_n\}$, if $y_1 \succ_D y_2$. Similarly, an S-type tier structure defines the order of magnitudes of the \emph{stochastic} mass-action rate functions of reactions with different source complexes. The reason why in Definition~\ref{def31} the quantity $(x_n\vee 1)^y$ appears, rather than the mass-action rate $x_n^y$, is technical: this is due to the necessity of considering sequences of points $\{x_n\}$ with possibly null components.
}

\begin{defn}\label{def_tier-seqence}
Let $X$ be the Markov process associated with a mass action system $(\S,\C,\Re, \kappa)$, with $X(0)=x_0$. Let $\{x_n\}$ be a sequence in $\Z^d_{\ge 0}$, with $n$ varying from 1 to $\infty$. The sequence $\{x_n\}$ is called a  \emph{proper tier sequence} of $X$ if the following holds:
\begin{enumerate}[label=(\roman*)]
\item $x_n \in \state$ for each $n\ge 1$; 
\item $\lim_{n\rightarrow\infty}x_{n,i} \in [0,\infty]$ for each $i$, and there exists an $i$ for which $\lim_{n\rightarrow \infty}x_{n,i} = \infty$;
\item $\C$ has both a D-type partition and an S-type partition along $\{x_n\}$. 
\end{enumerate}
\end{defn}

\subsection{Known results}

\reply{The next lemma points out that any unbounded sequence of points contains a proper tier sequence as a subsequence.}

\begin{lem}\label{lem21} Let $X$ be the Markov process associated with a mass action system $(\S,\C,\Re, \kappa)$, with $X(0)=x_0$. Let $\{x_n\}_{n=1}^\infty \subset \state$ be a sequence such that $\limsup_{n\rightarrow \infty}|x_n|=\infty$. Then there exists a subsequence of $\{x_n\}$ which is a proper tier sequence of $X$.
\end{lem}

The proof of Lemma \ref{lem21} is similar to that of Lemma 4.2 in \cite{AndGAC_ONE2011}. In particular, the relevant
tiers can each be constructed sequentially by repeatedly taking subsequences.
 

%
The following lemma states that for each proper tier sequence, there exists at least one reaction whose source and product complex are not in the same D-type tier.  This lemma was originally stated in \cite{ACKN2018} as Lemma 4.1.

\begin{lem}\label{lem: Tdi exists along tier sequence}
Let $X$ be the Markov process associated with a mass action system $(\S,\C,\Re, \kappa)$, with $X(0)=x_0$. For a proper tier sequence $\{x_n\}$ of $X$, there exists a reaction $y \to y' \in \Re$ such that $y \in T^{D,i}_{\{x_n\}}$ and $y' \in T^{D,j}_{\{x_n\}}$ for some $i\neq j$. 
\end{lem}

\subsection{Tiers along sequences of paths}\label{sec:along_path}

Here, we define tier structures on a set of reactions, which will be utilized to describe the behavior of the embedded Markov chain after a fixed amount of transitions. \reply{We recall here that the classical strategy of studying the behaviour of the embedded chain after each step, using the function $V$ described in \eqref{eq:MainLyapunov} as potential Lyapunov function, may not lead to a proof of positive recurrence. This is the case of Example~\ref{ex}, and this is the reason why we extend the method to consider mutliple transitions of the embedded chain.}

\begin{defn}\label{def:new tiers}
Let $X$ be the Markov process associated with a mass action system $(\S,\C,\Re, \kappa)$, with $X(0)=x_0$. Let \reply{$k$ be a finite, positive integer and let} $R=\{y_1\to y'_1,y_2\to y'_2,\dots,y_k \to y'_k \}$ be \reply{a sequence of elements of $\Re$ of length $k$}. Let $\{x_n\}$ be a sequence such that $\{x_n + \sum_{j=1}^{i-1}(y'_j -y_j)\}$ is a proper tier sequence of $X$ for each $i=1,2,\dots,k.$ Then, 
\begin{enumerate}[label=(\roman*)]
\item $R \in \mathbb{D}_{\{x_n\}}$ if $y_i \in T^{D,1}_{\{x_n+\sum_{j=1}^{i-1}(y'_j-y_j)\}}$ for each $i=1,2,\dots,k$ and $y'_\ell \notin T^{D,1}_{\{x_n+\sum_{j=1}^{\ell-1}(y'_j-y_j)\}}$ for some $\ell \in \{1,2,\dots,k\}$;
\item\label{part:most_likely} $R\in \mathbb{T}^{S,1}_{\{x_n\}}$ if $y_i \in T^{S,1}_{\{x_n+\sum_{j=1}^{i-1}(y'_j-y_j)\}}$ for each $i=1,2,\dots,k$.
\end{enumerate}
 \end{defn}

Note that \ref{part:most_likely} in Definition \ref{def:new tiers} implies that this particular sequence of reactions is one of the most likely sequences of $k$ reactions that will be observed if the process starts at state $x_n$. This will be made precise in the next lemma.  First, we introduce the following notation: consider a mass action system $(\S,\C,\Re, \kappa)$ and let $\tilde X$ be the associated embedded Markov chain, with $\tilde{X}_0=x$. Denote the event that the reactions $y_1\to y_1'$, $y_2 \to y_2'$, $\dots$, $y_k \to y_k'$ are the first $k$ reactions to occur \emph{in order} by 
\[
    x \xrightarrow{y_1\to y_1',\dots,y_k\to y'_k} x+\sum_{j=1}^k(y'_j-y_j).
\]

\begin{lem}\label{lem : prob of r jumps}
Let $X$ be the Markov process associated with a weakly reversible mass action system $(\S,\C,\Re, \kappa)$, with $X(0)=x_0$.  Let $\tilde{X}$ be the embedded discrete-time Markov chain of $X$. \reply{Let $k$ be a finite, positive integer and let $R=\{y_1\to y'_1$,$y_2\to y'_2,\dots,y_k\to y'_k\}$ be a sequence of elements of $\Re$ of length $k$.}  Let $\{x_n\}$ be a sequence such that $\{x_n + \sum_{j=1}^{\ell-1} (y'_j-y_j)\}$ is a proper tier sequence of $X$ for each $\ell=1,2,\dots,k.$ Then 
\begin{enumerate}[label=(\roman*)]
\item\label{part:positive_limit} if
 $R \in \mathbb{T}^{S,1}_{\{x_n\}},$ 
 \[\lmt P_{x_n}\left(x_n \xrightarrow{y_1\to y_1',\dots,y_k\to y'_k} x_n+\sum_{j=1}^k(y'_j-y_j) \right) > 0, \quad \text{and}\]
 \item\label{part:zero_limit} if 
 $R \not \in \mathbb{T}^{S,1}_{\{x_n\}},$ 
 \[\lmt P_{x_n}\left(x_n \xrightarrow{y_1\to y_1',\dots,y_k\to y'_k} x_n+\sum_{j=1}^k(y'_j-y_j)\right)=0.\]
 \end{enumerate}
\end{lem}

\begin{proof}
For each $n\geq 1$ and $1\leq m\leq k$, let $z_n(m)=x_n+\sum_{j=1}^{m-1}(y'_j-y_j)$. 
\reply{By \eqref{eq:transition prob1}, we have}
\begin{equation}\label{eq:product}
P_{x_n}\left (x_n \xrightarrow{y_1\to y_1',\dots,y_k\to y'_k} x_n+\sum_{j=1}^k(y'_j-y_j)\right )
= \prod_{m=1}^{k} \frac{\lambda_{y_m\to y'_m}(z_n(m))}{\overline{\lambda}(z_n(m))}\mathbbm{1}_{\{\overline{\lambda}(z_n(m))>0\}}.
\end{equation}

Assume that $R \in \mathbb{T}^{S,1}_{\{x_n\}}$. Then, for each $1\leq m \leq k$ we have $y_m \in T^{S,1}_{\{z_n(m)\}}$, and by definition of S-type tiers 
\begin{align*}
\lmt \dfrac{\lambda_{y_m\to y'_m}(z_n(m))}{\overline{\lambda}(z_n(m))} = \lmt \dfrac{\kappa_{y_m\to y'_m}}{\dsum_{y\to y'\in \Re}\kappa_{y\to y'}\lambda_y(z_n(m))/\lambda_{y_m}(z_n(m))} > 0.
\end{align*} 
Hence, \ref{part:positive_limit} holds.

Now assume that $R \not \in \mathbb{T}^{S,1}_{\{x_n\}}$.  \reply{Hence, by the definition $\mathbb{T}^{S,1}_{\{x_n\}}$} there exists an $1\leq \ell \le k$ with $y_\ell \in T^{S,j}_{\{z_n(\ell)\}}$ for $j > 1$.  We now consider the following single term from \eqref{eq:product}
\begin{equation}\label{eq:5678765}
    \frac{\lambda_{y_\ell\to y'_\ell}(z_n(\ell))}{\overline{\lambda}(z_n(\ell))}\mathbbm{1}_{\{\overline{\lambda}(z_n(\ell))>0\}}.
\end{equation}
There are two cases.  If all  intensity functions are identically equal to zero at $z_n(\ell)$, then the term \eqref{eq:5678765} is zero.  If not all of the intensity functions are zero, then, because the network is weakly reversible, there is a source complex $y\in T^{S,1}_{\{z_n(\ell)\}}$, in which case
\reply{\[
\lim_{n\to \infty} \frac{\lambda_{y_\ell\to y'_\ell}(z_n(\ell))}{\overline{\lambda}(z_n(\ell))}\mathbbm{1}_{\{\overline{\lambda}(z_n(\ell))>0\}} \le \lim_{n\to \infty} \frac{\lambda_{y_\ell\to y'_\ell}(z_n(\ell))}{\lambda_{y\to y'}(z_n(\ell))}= 0 \]}
by definition of S-type tiers. Hence, \ref{part:zero_limit} is shown since all the terms in \eqref{eq:product} are less than or equal to one.
\end{proof}

Further useful properties of tiers along sequences of paths are given in the following lemmas.

\begin{lem}\label{lem : Td1 stays same after jumps}
Let $X$ be the Markov process associated with a mass action system $(\S,\C,\Re, \kappa)$, with $X(0)=x_0$. Suppose that $w$ is a fixed vector and that both $\{x_n\}$ and $\{x_n+w\}$ are proper tier sequences for $X$. 
Suppose their D-type tier partitions are $\{T^{D,j}_{\{x_n\}}\}_{j=1}^{J_1}$ and $\{T^{D,j}_{\{x_n+w\}}\}_{j=1}^{J_2}$, respectively.  Then $J_1 =J_2$ and $T^{D,j}_{\{x_n\}} = T^{D,j}_{\{x_n+w\}}$ for each $j$. 
\end{lem}	
\begin{proof}
To show that the D-type partitions along $\{x_n\}$ and $\{x_n+w\}$ coincide, it suffices to show that if $y\succ_D y'$ for the D-type partitions along $\{x_n\}$, then $y\succ_D y'$ for the D-type partitions along $\{x_n+w\}$. This holds because the limits defined in part 2 of Definition \ref{def31} do not change if $x_n$ is replaced by $x_n+w$.
\end{proof}
\begin{lem}\label{lem : Ts1 after jump}
Let $X$ be the Markov process associated with a mass action system $(\S,\C,\Re, \kappa)$, with $X(0)=x_0$. Let $y\to y' \in \Re$ and assume that both $\{x_n\}$ and $\{x_n+y'-y\}$ are proper tier sequences of $X$. Moreover, assume that $y \not \in \Tsinf$ and $y' \in \Td1$. Then $y' \in T^{S,1}_{\{x_n+y'-y\}}$.
\end{lem}	

\reply{
To prove Lemma~\ref{lem : Ts1 after jump} we will make use of the following result, which follows immediately from  \cite[Corollary 7]{AK2018}.

\begin{lem}\label{lem:alghkadjbhvjd}
    Let $\{x_n\}$ be a proper tier-sequence of a reaction network $(\S,\C,\Re)$.  Suppose $y \notin T^{S,\infty}_{\{x_n\}}$ and $y \in T^{D,1}_{\{x_n\}}$. Then $y \in T^{S,1}$ and $\lim_{n\to \infty} \lambda_y(x_n) = \infty$.
\end{lem}

We are now ready to prove Lemma~\ref{lem : Ts1 after jump}.}
\begin{proof}[Proof of Lemma~\ref{lem : Ts1 after jump}]
 We have that $y \not \in \Tsinf$, which  is equivalent to $x_n \ge y$, which is in turn equivalent to $x_n + y' - y \ge y'$, which implies  $y' \not \in T^{S,\infty}_{\{x_n+y'-y\}}$. Now we must just show that $y'$ actually yields one of the largest intensities on the sequence $\{x_n+y'-y\}$.

 Because $y' \in \Td1$,  Lemma \ref{lem : Td1 stays same after jumps} implies $y' \in T^{D,1}_{\{x_n+y'-y\}}$. Since $y' \notin T^{S,\infty}_{\{x_n+y'-y\}}$,  Lemma~\ref{lem:alghkadjbhvjd} implies $y'\in T^{S,1}_{\{x_n+y'-y\}}$, and the result is shown.
\end{proof}

\section{Tiers and positive recurrence}\label{sec:positive_rec}

The goal of this section is to prove the following theorem.

\begin{thm}\label{thm:main2}
Consider a weakly reversible mass action system $(\S,\C,\Re, \kappa)$ that has a single linkage class, and let $X$ be the associated continuous time Markov chain, with $X(0)=x_0$. Suppose 
\begin{equation}\label{hypo1}
T^{S,1}_{\{z_n\}} \subset T^{D,1}_{\{z_n\}}
\quad\text{for any proper tier sequence $\{z_n\}$ of $X$.}
\end{equation}
Then $X$ is positive recurrent.
\end{thm}

To prove this result, we require a number of preliminary lemmas.

\begin{lem}\label{lemma:main}
Let $\{x_n\}$ be a sequence of points in $\Z^d_{\ge 0}$ be a sequence such that $\lim_{n\to \infty} |x_n|=\infty$.  Let $I	= \{ i \ |\  x_{n,i} \rightarrow \infty,  \ \ \textrm{as} \ \ n \rightarrow \infty \}$ and assume that $I \ne \emptyset$ and that $I^c = \{i \ | \ \sup_{n} x_{n,i} < \infty\}.$
Let $\eta \in \R^d$ be such that $x_n + \eta \in \R^d_{\ge 0}$ for each $n \ge 1$, and let $V$ be defined as in \eqref{eq:MainLyapunov}. Then there exists a positive constant $C$ such that for all $n$
\begin{align*}
  V(x_n+\eta)-V(x_n) \le \ln{( (x_n\vee 1)^\eta )}+C.
\end{align*} 
 \end{lem}
\begin{proof} 
Since the terms associated with $I^c$ are uniformly bounded,
there exists a $C_1\ge 0$ such that
\begin{align}
V(x_n+\eta)-V(x_n) &\le \dsum_{i\in I}[ (x_{n,i}+\eta_i)(\ln{(x_{n,i}+\eta_i)}-1) - x_{n,i}(\ln{(x_{n,i})}-1)] + C_1\notag\\
&=  \dsum_{i\in I}\left[ x_{n,i}\ln \left(1+\frac{\eta_i}{x_{n,i}}\right)-\eta_i + \eta_i\ln{(x_{n,i}+\eta_i)}\right] + C_1\notag.
\end{align}
 Using the fact that $\lim_{t \to \infty} (1 + \tfrac{\alpha}{t})^t = e^\alpha$ for any $\alpha$, we have  that
\[
x_{n,i}\ln \left(1+\frac{\eta_i}{x_{n,i}}\right)-\eta_i \rightarrow 0, \ \ \ \textrm{as} \ \ n \rightarrow \infty,
\]
 for each $i\in I$.
Hence, there is a  $C_2>0$ for which 
\[
V(x_n+ \eta)-V(x_n) \le \dsum_{i \in I} \eta_i\ln (x_{n,i}+\eta_i) + C_2,
\]
for each $n$.  The limit
\[
	\ln \left[ (x_{n,i}+\eta_i)^{\eta_i}\right] - \ln(x_{n,i}^{\eta_i}) \to 0, 
\]
as $n\to \infty$ then implies the existence of a $C_3>0$  for which 
\begin{align*}
V(x_n+\eta)&-V(x_n)   \le \sum_{i \in I}\ln \left(  x_{n,i}^{\eta_i}\right) + C_3 = \ln \left(  \prod_{i\in I} x_{n,i}^{\eta_i}\right) + C_3.
\end{align*}
Finally, we may add back in the bounded terms to conclude that there is a $C>0$ for which 
\begin{align*}
V(x_n+\eta)&-V(x_n) \le   \ln{\left(\dprod_{i=1}^d (x_{n,i} \vee 1)^{\eta_i}\right)} + C = \ln{(x_n\vee 1)^\eta} +C,
\end{align*}
and the proof is complete.
\end{proof}

\begin{lem}\label{lem : limit of prob x difference V}
	Let $X$ be the Markov process associated with a weakly reversible mass action system $(\S,\C,\Re, \kappa)$, with $X(0)=x_0$. Suppose \eqref{hypo1} holds and let
  $\tilde{X}$ be the embedded discrete-time Markov chain of $X$. \reply{Let $k$ be a finite, positive integer and let $R=\{y_1\to y'_1$,$y_2\to y'_2,\dots,y_k\to y'_k\}$ be a sequence of elements of $\Re$ of length $k$.}  \reply{Let $\{x_n\}$ be a fixed sequence such that $\{x_n + \sum_{j=1}^{i-1}(y'_j-y_j)\}$ is a proper tier sequence of $X$ for each $i=1,2,\dots,k$.
  }
  Then, for the function $V$ defined as \eqref{eq:MainLyapunov}
\begin{enumerate}[label=(\roman*)]
\item\label{part:bounded_increment} there is a constant $K>0$ such that
\begin{align*}
 P_{x_n}\left (x_n \xrightarrow{y_1\to y_1',\dots,y_k\to y'_k} x_n+\sum_{j=1}^k(y'_j-y_j)\right )\left(V\left(x_n+\sum_{j=1}^k (y'_j-y_j)\right)-V(x_n)\right) \le K,
\end{align*}
for all $n \ge 1$, and
\item\label{part:infinite_increment} if $R \in  \mathbb{T}^{S,1}_{\{x_n\}}\cap \mathbb{D}_{\{x_n\}}$, then
\begin{align*}
\hspace{-0.4in}\lmt P_{x_n}\left (x_n \xrightarrow{y_1\to y_1',\dots,y_k\to y'_k} x_n+\sum_{j=1}^k(y'_j-y_j)\right )\left(V\left(x_n+\sum_{j=1}^k (y'_j-y_j)\right)-V(x_n)\right) = -\infty.
\end{align*}
\end{enumerate}
\end{lem}
\begin{proof}
For each $n\geq 1$ and $1\leq m\leq k$, let $z_n(m)=x_n+\sum_{j=1}^{m-1}(y'_j-y_j)$. Note that each $\{z_n(m)\}$ is a proper tier sequence, which implies $z_n(m) \in \R^d_{\ge 0}$ for each $n$ and $m$.  Hence, Lemma \ref{lemma:main} implies that for each proper tier sequence $\{x_n\}$, there is a constant $C>0$ such that for all $n$
\reply{\begin{align}\label{eq : asymptotic of difference of V}
V\left(x_n+\sum_{m=1}^k (y'_j-y_j)\right)-V(x_n) &= \sum_{m=1}^k\Big ( V(z_n(m)+y'_m-y_m)-V(z_n(m)) \Big) \notag\\
&\le \ln{\left ( \prod_{m=1}^k \frac{(z_n(m) \vee 1)^{y'_m}}{(z_n(m) \vee 1)^{y_m}} \right ) }+ C. 
\end{align}}
First we suppose $y_\ell \in T^{S,\infty}_{\{z_n(\ell)\}}$ for some $\ell$, then by \eqref{eq:product}
\begin{align*}
P_{x_n}\left (x_n \xrightarrow{y_1\to y_1',\dots,y_l\to y'_k} x_n+\sum_{j=1}^k(y'_j-y_j)\right )=0,
\end{align*}
and \ref{part:bounded_increment} is shown.

Now assume \reply{$y_m \not \in T^{S,\infty}_{\{z_n(m)\}}$} for all $1\le m\le k$. 
Then, there is at least one complex \reply{$\hat y_m\in T^{S,1}_{\{z_n(m)\}}$} for each $1\le m\le k$. By \eqref{hypo1} and  \cite[Lemma 6]{AK2018}, for any ${y}\in \C$ and any $1\leq m\leq k$
\begin{equation}
\lmt \frac{(z_n(m)\vee 1)^{{y}}}{\lambda_{\hat y_m}(z_n(m))}
= \lmt \frac{(z_n(m)\vee 1)^{{y}}}{(z_n(m)\vee 1)^{\hat y_m}}\frac{(z_n(m)\vee 1)^{\hat y_m}}{\lambda_{\hat y_m}(z_n(m))} < \infty.  \label{eq:bound of psi and phi}
\end{equation}
By \eqref{eq:product} and \eqref{eq : asymptotic of difference of V}, we have
\begin{align}
P_{x_n}\bigg (x_n& \xrightarrow{y_1\to y_1',\dots,y_l\to y'_k} x_n+\sum_{j=1}^k(y'_j-y_j)\bigg )\left(V\left(x_n+\sum_{j=1}^k (y'_j-y_j)\right)-V(x_n)\right) \notag\\
\le & 
\left(\prod_{m=1}^{k}\frac{\kappa_{y_m\to y'_m} \lambda_{y_m}(z_n(m))}{\overline{\lambda}(z_n(m))}\right)\left ( \ln{\left ( \prod_{m=1}^k \frac{(z_n(m) \vee 1)^{y'_m}}{(z_n(m) \vee 1)^{y_m}} \right ) }+ C \right )\notag\\
= &\left(\prod_{m=1}^k\frac{\kappa_{y_m\to y'_m} \lambda_{y_m}(z_n(m))}{(z_n(m)\vee 1)^{y_m}}\right) \psi_n\left(\ln{\left(\frac{1}{\psi_n}\right)}+\ln{(\phi_n)}+C\right),
\label{eq : prob x lya}  
\end{align} 
where
\[ \psi_n = \prod_{m=1}^k \frac{(z_n(m) \vee 1)^{y_m}}{\overline{\lambda}(z_n(m))} \quad \text{and} \quad \phi_n =\prod_{m=1}^k \frac{(z_n(m) \vee 1)^{y'_m}}{\overline{\lambda}(z_n(m))},
\]
and we recall the definition of $\overline \lambda$ from \eqref{lambdabar}.
Since the network is weakly reversible, any complex is a source complex, which in turn implies that $\lambda_{y}(x)\leq \overline{\lambda}(x)$ for any $y\in\C$ and $x\in\Z^d$. By combining this with \eqref{eq:bound of psi and phi}, it follows that there is a constant $C''>0$ such that
\begin{align}
 0<\psi_n \le C'' \quad \text{and} \quad   
 0<\phi_n \le C''    \quad \text{for all} \ \ n. \label{eq : psi and phi}
\end{align}
Hence, \ref{part:bounded_increment} is proven by combining \eqref{eq : prob x lya}, \eqref{eq : psi and phi}, the fact that $\psi\ln{(1/\psi)}$ is bounded from above if $\psi$ is bounded from above, and finally the fact that \reply{$\lambda_y(x)\le x^y \le (x \vee 1)^y$} for any $y\in\C$ and $x\in\Z^d$.

In order to prove \ref{part:infinite_increment}, suppose $R \in  \mathbb{T}^{S,1}_{\{x_n\}}\cap \mathbb{D}_{\{x_n\}}$.   Since $R \in \mathbb{T}^{S,1}_{\{x_n\}}$, Lemma \ref{lem : prob of r jumps} implies 
\[
    \lmt P_{x_n}\left (x_n \xrightarrow{y_1\to y_1',\dots,y_k\to y'_k} x_n+\sum_{j=1}^k(y'_j-y_j)\right )>0.
    \]
    Hence, \ref{part:infinite_increment} is shown if we prove the right-hand side of \eqref{eq : asymptotic of difference of V} goes to $-\infty$, as $n \to \infty.$
Because $R \in \mathbb{D}_{\{x_n\}}$, which implies each \reply{$y_m \in T^{D,1}_{\{z_n(m)\}}$},
we may conclude that each term of the form
\[
 \frac{(z_n(m) \vee 1)^{y'_m}}{(z_n(m) \vee 1)^{y_m}}
\]
is uniformly bounded from above in $n$.  Moreover, and again because $R \in \mathbb{D}_{\{x_n\}}$, there exists $1\leq \ell\leq k$ such that $y_\ell \in T^{D,1}_{z_n(m)}$, but \reply{$y'_\ell \not \in T^{D,1}_{\{z_n(m)\}}$}.  Hence, 
\begin{align*}
\lmt \ln{\left ( \prod_{m=1}^k \frac{(z_n(m) \vee 1)^{y'_m}}{(z_n(m) \vee 1)^{y_m}} \right ) } = -\infty,
\end{align*}
and \ref{part:infinite_increment} is shown.
\end{proof}

\begin{lem}\label{thm : main analytic thm}
Let $k \in \Z_{>0}$ be fixed and let $X$ be the Markov process associated with a weakly reversible mass action system $(\S,\C,\Re, \kappa)$, with $X(0)=x_0$. Suppose \eqref{hypo1} holds and that for any 
 proper tier sequence $\{x_n\}$ of $X$, there is a sequence of $k$ reactions $R=\{y_1\to y'_1,\dots,y_k \to y'_k\}$  
 such that, by potentially considering a subsequence of $\{x_n\}$, the following holds:
\begin{enumerate}[label=(\roman*)]
\item\label{part:nothing} $\{x_n+\sum_{j=1}^{m-1}(y'_j-y_j)\}$ is a proper tier sequence for each $m=1,2,\dots,k$, and 
\item\label{part:good_sequence} $R \in \mathbb{T}^{S,1}_{\{x_n\}} \cap \mathbb{D}_{\{x_n\}}$.
\end{enumerate}
Then $X$ is positive recurrent.
\end{lem}
\begin{proof}
Let $\tilde X_n$, $n\ge 0$, be the embedded chain for $X$.  For $n \ge 0$, define $Z^k_n = \tilde X_{nk}$.  Note that positive recurrence of  $Z^k_n$ implies positive recurrence of $\tilde X_n$, which in turn implies positive recurrence of $X$ via Theorem \ref{thm : relation between conti and disc}. 

Suppose, in order to find a contradiction, that $Z^k_n, n \ge 0,$ is not positive recurrent for some choice of $Z^k_0 = \tilde X_0 = x_0$.  Then, by the usual Foster-Lyapunov criteria \cite{MT-LyaFosterIII}, there exists a sequence of points $x_n\in\state$ with $\lim_{n\to\infty}|x_n| = \infty$ and for which 
\begin{equation}\label{eq:589707}
	E_{x_n}[V(Z^k_{1})] - V(x_n) > -1,
\end{equation}
where $V$ is defined in \eqref{eq:MainLyapunov}.
By Lemma \ref{lem21}, we may choose a subsequence of $\{x_n\}$ that is a proper tier sequence of $X$.  By assumption, there is then a further subsequence, which we will also denote by $\{x_n\}$, and an ordered sequence of reactions $\{y_1 \to y_1', \dots, y_k \to y_k'\}$ for which both \ref{part:nothing} and \ref{part:good_sequence} in the statement of the lemma hold.  Denote by $\Re_k$ all possible sequences of reactions of length $k$ chosen from $\Re$.  We will denote  elements of $\Re_k$ by $\bar{y}_1\to\bar{y}'_1, \dots, \bar{y}_k\to\bar{y}'_k$.  Perhaps after consideration of another subsequence, by applying \lemref{lem21} we may further assume that $\{x_n+\sum_{j=1}^{m-1}(\bar y'_j-\bar y_j)\}$ is a proper tier sequence for each $m=1,2,\dots,k$ and for each element of $\Re_k$, which are finitely many.

%

We may now observe that 
\begin{align*}
E_{x_n}&[V(Z^k_{1})] - V(x_n)\\
= &\sum_{\Re_k} P_{x_n}\left(x_n \xrightarrow{\bar{y}_1\to\bar{y}'_1, \dots, \bar{y}_k\to\bar{y}'_k} x_n+\sum_{j=1}^k(\bar{y}'_j-\bar{y}_j) \right)\left( V\left(x_n+\sum_{j=1}^k(\bar{y}'_j-\bar{y}_j)\right)-V(x_n) \right),
\end{align*}
where $P_{x}$ is a probability measure of $\tilde X$ with an initial point $x$.
Note that by \lemref{lem : limit of prob x difference V}, there exists a constant $\overline K$ such that 
\begin{align*}
&\sup_n P_{x_n}\left(x_n \xrightarrow{\bar{y}_1\to\bar{y}'_1, \dots, \bar{y}_k\to\bar{y}'_k} x_n+\sum_{j=1}^k(\bar{y}'_j-\bar{y}_j) \right)\left(V\left(x_n+\sum_{j=1}^k(\bar{y}'_j-\bar{y}_j)\right)-V(x_n)\right) \le \overline K,
\end{align*}
for each element of $\Re_k$.  Note that the particular $K$ from the conclusion of Lemma \ref{lem : limit of prob x difference V}  depends upon the particular element of $\Re_k$ under consideration.  Therefore, $\overline K$ is the maximum of all such $K$ taken over these elements, which are finitely many.  Next, note that for the particular sequence $y_1\to y_1', \dots, y_k \to y_k'$ satisfying \ref{part:nothing} and \ref{part:good_sequence}, Lemma \ref{lem : limit of prob x difference V} implies that
\begin{equation*}
\lim_{n\to \infty}  P_{x_n}\left(x_n \xrightarrow{{y}_1\to{y}'_1, \dots, {y}_k\to{y}'_k} x_n+\sum_{j=1}^k({y}'_j-{y}_j) \right)\left(V\left(x_n+\sum_{j=1}^k({y}'_j-{y}_j)\right)-V(x_n)\right) =-\infty.
\end{equation*}
We immediately see that the above three equations  contradict \eqref{eq:589707}, and so the result is shown.
\end{proof}

The next lemma states that assumption \eqref{hypo1} actually ensures that conditions \ref{part:nothing} and \ref{part:good_sequence} in Lemma \ref{thm : main analytic thm} hold.

\begin{lem}\label{lem:existence of a set of source in D and Ts1 }
Consider a weakly reversible mass action system $(\S,\C,\Re, \kappa)$ that has a single linkage class, and let $X$ be the associated continuous time Markov chain. Let $X(0)=x_0$ and assume $x_0$ is not an absorbing state. Let $r=|\Re|$ and suppose \eqref{hypo1} holds. 
Let $\{x_n\}$ be a proper tier sequence of $X$. Then, by potentially considering a sub-sequence of $\{x_n\}$, there exists a sequence of reactions $R=\{y_1\to y'_1,y_2\to y'_2,\dots,y_r \to y'_r \}$ such that
\begin{enumerate}[label=(\roman*)]
\item\label{part:nothing2} $\{x_n+\sum_{j=1}^{m-1}(y'_j-y_j)\}$ is a proper tier sequence for each $m=1,2,\dots,r$, and 
\item\label{part:good_sequence2} $R \in \mathbb{T}^{S,1}_{\{x_n\}} \cap \mathbb{D}_{\{x_n\}}$.
\end{enumerate}
\end{lem}
\begin{proof}
Since $x_0$ is not absorbing and the network is weakly reversible, it follows from Lemma~\ref{lem : minimum of intensities} that for any $x\in \state$ there exists a reaction $y\to y'$ with $\lambda_{y\to y'}(x)>0$, hence $T^{S,1}_{\{x_n\}}\neq \emptyset$. Let $y^\star\in T^{S,1}_{\{x_n\}}$, which by assumption is contained in $\Td1$. Moreover, by \lemref{lem: Tdi exists along tier sequence} there exists $y^{\star\star}\in \C\setminus \Td1$. Since the network $(\S, \C, \Re)$ is weakly reversible and $y^\star$ and $y^{\star\star}$ are in the same linkage class, there exists a sequence of reactions $\{y_1\to y_1',\dots,y_H\to y'_H\}$ with $y_1=y^\star\in T^{S,1}_{\{x_n\}}$, $y'_j=y_{j+1}$ for all $1\leq j<H$, and $y'_H=y^{\star\star}\notin\Td1$.
It follows from \lemref{lem21} that by potentially considering a subsequence of $\{x_n\}$, we may assume that $\{x_n+\sum_{j=1}^{m}(y'_j-y_j)\}$ is a proper tier sequence of $X$ for each $1\leq m\leq H$. It follows from \lemref{lem : Td1 stays same after jumps} that $y'_H\notin T^{D,1}_{\{x_n+\sum_{j=1}^{H-1} (y'_j-y_j)\}}$. Let $h$ be the smallest index such that $y'_h\notin T^{D,1}_{\{x_n+\sum_{j=1}^{h-1} (y'_j-y_j)\}}$. Note that $1\leq h\leq H\leq r$. Define
$$\tilde R=\{y_1\to y_1',\dots,y_h\to y_h'\}.$$
By \lemref{lem : Ts1 after jump}, for each $2\leq m\leq h$
$$y_m=y'_{m-1}\in T^{S,1}_{\{x_n+\sum_{j=1}^{m-1} (y'_j-y_j)\}},$$
which implies that $\tilde R\in\mathbb{T}^{S,1}_{\{x_n\}}$. Moreover, $\tilde R\in\mathbb{D}_{\{x_n\}}$ by \eqref{hypo1} and by definition of $h$.

If $h=r$, then the the proof is complete by choosing $R=\tilde R$.

If $h<r$, then we may recursively add reactions to $\tilde R$ as follow: since there are finitely many reactions and there is no absorbing state in $\state$, by potentially considering a subsequence of $\{x_n\}$ we may assume that there exists a reaction $\hat y_{h+1}\to \hat y'_{h+1}$ such that 
$$\lambda_{\hat y_{h+1}\to \hat y'_{h+1}}\left(x_n+\sum_{j=1}^{h} (y'_j-y_j)\right)=\max_{y\to y'\in\Re}\lambda_{y\to y'}\left(x_n+\sum_{j=1}^{h} (y'_j-y_j)\right)>0.$$
It follows that $\hat y_{h+1}\in T^{S,1}_{\{x_n+\sum_{j=1}^{h} (y'_j-y_j)\}}$, hence by \eqref{hypo1}
$$\{y_1\to y_1',\dots,y_h\to y_h', \hat y_{h+1}\to \hat y'_{h+1}\}\in \mathbb{T}^{S,1}_{\{x_n\}} \cap \mathbb{D}_{\{x_n\}}.$$
Moreover, by Lemma~\ref{lem21}, there exists a subsequence of $\{x_n\}$ such that
$$\left\{x_n+\sum_{j=1}^{h} (y'_j-y_j)+ \hat y'_{h+1}-\hat y_{h+1}\right\}$$
is a proper tier sequence. If $h+1=r$, the proof is complete. 
Otherwise, the argument can be iterated with the new set of reactions, until the cardinality of the sequence of interest is $r$. This concludes the proof.
\end{proof}

We are now ready to prove Theorem~\ref{thm:main2}.

\begin{proof}[Proof of Theorem~\ref{thm:main2}]
 If $x_0$ is an absorbing state, then $X$ is positive recurrent and the proof is complete. If $x_0$ is not absorbing, then \lemref{lem:existence of a set of source in D and Ts1 } can be applied. Therefore, the proof is concluded by \lemref{thm : main analytic thm}, by choosing $k=r$.
\end{proof}

\subsection{Structural network conditions}\label{subsec2}

In this section, we prove Theorem~\ref{thm:main}. We do so by showing that under the structural assumptions in the statement of Theorem~\ref{thm:main}, Theorem~\ref{thm:main2} can be applied. 

\begin{proof}[Proof of Theorem~\ref{thm:main}]
 If $x_0$ is an absorbing state, then $X$ is positive recurrent and the proof is concluded. Now assume that $x_0$ is not absorbing. In order to apply Theorem~\ref{thm:main2} and conclude that $X$ is positive recurrent, we only need to show that under the assumptions of Theorem~\ref{thm:main}, \eqref{hypo1} holds. By Lemma~\ref{lem : minimum of intensities}, for any $x\in \state$ there exists a reaction $y\to y'$ with $\lambda_{y\to y'}(x)>0$, hence $T^{S,1}_{\{x_n\}}\neq \emptyset$ for any proper tier sequence $\{x_n\}$. Fix any $y^\star\in \Ts1$, $y^{\star\star}\in\Td1$. We will conclude the proof by showing that necessarily
 \begin{equation}\label{eq:dkkjgsdjfhskfhsk}
     y^\star  \sim_D y^{\star\star},
 \end{equation}
 which in turn implies $y^\star\in \Td1$ and hence \eqref{hypo1}. We consider two cases: first, assume that $y^{\star\star}\notin \Tsinf$. We have
 \begin{equation}\label{eq:showing y and y tilde are in same tier}
 \frac{(x_n\vee 1)^{y^{\star\star}}}{(x_n \vee 1)^{y^\star}} = 
 \frac{(x_n\vee 1)^{y^{\star\star}}}{\lambda_{y^{\star\star}}(x_n)}\cdot
 \frac{\lambda_{y^{\star\star}}(x_n)}{\lambda_{y^\star}(x_n)}\cdot
 \frac{\lambda_{y^\star}(x_n)}{(x_n\vee 1)^{y^\star}}.
\end{equation}
By \cite[Corollary 7]{AK2018}, we have $y^{\star\star}\in \Ts1$. Hence the second term in the right-hand side of \eqref{eq:showing y and y tilde are in same tier} tends to a positive constant as $n\to\infty$, and by \cite[Lemma 6]{AK2018} the first and the third terms in the right-hand side of \eqref{eq:showing y and y tilde are in same tier} tend to 1 as $n\to\infty$. Hence, the quantity \eqref{eq:showing y and y tilde are in same tier} tends to a positive constant as $n\to\infty$, which implies \eqref{eq:dkkjgsdjfhskfhsk}.

Now assume that $y^{\star\star}\in \Tsinf$, meaning that \begin{equation}\label{eq:zero}
 \lambda_{y^{\star\star}}(x_n)=0\quad\text{for all }n\geq1.
\end{equation}
Since $y^{\star\star}\in \Td1$, by \cite[Lemma 5]{AK2018}
\begin{equation}\label{eq:infinity}
    \lim_{n\to\infty} (x_n\vee 1)^{y^{\star\star}}=\infty.
\end{equation}
Since by assumption the network is binary, the only possibility for both \eqref{eq:zero} and \eqref{eq:infinity} to hold is that there exist $1\leq u,v\leq d$  with $u\neq v$ such that
\begin{enumerate*}[label=(\emph{\alph*}), before=\unskip{: }, itemjoin={{; }}, itemjoin*={{, and }}]
    \item $y^{\star\star}_u=y^{\star\star}_v=1$ and $y^{\star\star}_i=0$ for all $1\leq i\leq d$ with $i\notin\{ u,v\}$
    \item $x_{nu}=0$ for all $n\geq 1$
    \item $\lim_{n\to\infty} x_{nv}=\infty$
\end{enumerate*}.
By assumption, either $\tilde y=S_v \in \C$ or $\hat y=2S_v\in \C$, or both inclusions hold. If $\hat y$ were in $\C$, then we would have
$$\lim_{n\to \infty}\frac{(x_n\vee 1)^{y^{\star\star}}}{(x_n\vee 1)^{\hat y}}=\lim_{n\to \infty}\frac{x_{nv}}{x_{nv}^2}=0,$$
which would imply $\hat y\succ_D y^{\star\star}$. This is in contradiction with the assumption that $y^{\star\star}\in\Td1$, hence $\hat y\notin \C$ and $\tilde y\in\C$. Note that in this case $(x_n\vee 1)^{y^{\star\star}}=x_{nv}=\lambda_{\tilde y}(x_n)$ for all $n\geq1$. Hence,
$$\lim_{n\to\infty} \frac{(x_n\vee 1)^{y^{\star\star}}}{(x_n\vee 1)^{y^\star}}=\lim_{n\to\infty}\frac{\lambda_{\tilde y}(x_n)}{\lambda_{y^\star}(x_n)}\cdot \frac{\lambda_{y^\star}(x_n)}{(x_n\vee 1)^{y^\star}}<\infty,$$
because $y^\star\in\Ts1$ and the second term in the right-hand side limit tends to 1 by \cite[Lemma 6]{AK2018}. Then, either $y^\star\succ_D y^{\star\star}$ or \eqref{eq:dkkjgsdjfhskfhsk} holds. Since $y^{\star\star}\in\Td1$, the former cannot hold and the proof is concluded.
\end{proof}


	\section{Acknowledgements}
	We gratefully acknowledge the American Institute of Mathematics (AIM) for hosting our SQuaRE entitled ``Dynamical Properties of Deterministic and Stochastic Models of Reaction Networks'' where this work was initiated.
	
	DA was supported by Army Research Office grant  W911NF-18-1-0324.
	
	DC received funding from the European Research Council (ERC) under the European Union’s Horizon 2020 research and innovation programme grant agreement no. 743269 (CyberGenetics project).
	 \bibliographystyle{plain}
\bibliography{res}
	
\end{document}